\documentclass[]{amsart}
\usepackage{amsmath,amsfonts,amssymb,amsthm}
\usepackage{enumerate, color}
\usepackage{mathrsfs}
 \DeclareMathOperator{\ord}{ord}
 \DeclareMathOperator{\re}{Re}
\DeclareMathOperator{\id}{id} \DeclareMathOperator{\im}{Im}

\DeclareMathOperator{\cs}{CS}
\DeclareMathOperator{\res}{res}

\DeclareMathOperator{\sing}{Sing}
\DeclareMathOperator{\sat}{Sat}

\usepackage{graphicx}

\def\petal{\Delta_{\eta,\delta}^{k}}

\def\N{\mathbb N}
\def\R{\mathbb R}
\def\C{\mathbb C}
\def\Q{\mathbb Q}

\def\P{\mathbb P}

\def\Cd#1{(\C^#1,0)}

\newcommand\parcial[1]{\dfrac{\partial}{\partial #1}}
\def\N{\mathbb N}
\def\R{\mathbb R}
\def\C{\mathbb C}
\def\Q{\mathbb Q}

\def\P{\mathbb P}

\def\Cd#1{(\C^#1,0)}

\def\parcial#1{\frac{\partial}{\partial #1}}
\def\Parcial#1#2{\frac{\partial #1}{\partial #2}}

\theoremstyle{plain}
\newtheorem{theorem}{Theorem}[section]

\theoremstyle{definition}

\newtheorem{remark}[theorem]{Remark}
\newcommand{\obra}[3]{{\sc #1} {\em #2}. {#3}.}

\author{Lorena L\'opez-Hernanz}
\address{Lorena L\'opez-Hernanz\\ Departamento de F\'isica y Matem\'aticas, Universidad de Alcal\'a}
\email{lorena.lopezh@uah.es}
\date{}
\thanks{Partially supported by Ministerio de Ciencia e Innovación, Spain, PID2022-139631NB-I00}
\title[Dynamics of tangent to the identity biholomorphisms in dimension 2]{Local dynamics of tangent to the identity biholomorphisms in dimension two}
\begin{document}
\maketitle
\begin{abstract}
This survey provides an overview of the existence of stable manifolds for tangent to the identity biholomorphisms in $(\C^2,0$) and its connection with the existence of invariant curves for formal vector fields.
\end{abstract}

\renewcommand{\abstractname}{Résumé}
\begin{abstract}
Ce survey présente un aperçu de l'existence de variétés stables pour les biholomorphismes tangents à l'identité dans $(\C^2,0)$
et de son lien avec l'existence de courbes invariantes pour les champs de vecteurs formels.
\end{abstract}

\section{Introduction}
This survey text is based on a short course given at the conference ``Foliations, birational geometry and applications'' at CIRM in February 2025. Our aim is to provide an overview of the existence of stable manifolds for germs of tangent to the identity biholomorphisms in $(\C^2,0)$ and its connection with the existence of invariant curves for germs of formal vector fields. Section~\ref{sec:invariantcurves-vf} is devoted to formal vector fields in $(\C^2,0)$: we recall Seidenberg's reduction of singularities, Briot-Bouquet theorem on the existence of invariant curves when the vector field has a simple singularity and Camacho-Sad theorem about the existence of an invariant curve for every formal vector field in $(\C^2,0)$. Sections \ref{sec:HakimAbate} and \ref{sec:LRRS-LR} focus on the dynamics of tangent to the identity biholomorphisms in $(\C^2,0)$. In Section~\ref{sec:HakimAbate} we recall Écalle and Hakim theorem, which guarantees the existence of 1-dimensional stable manifolds when the biholomorphism has a non-degenerate characteristic direction, and Abate theorem, that ensures the existence of 1-dimensional stable manifolds for any tangent to the identity biholomorphism in $(\C^2,0)$ with an isolated fixed point. We relate these theorems to the results discussed in Section~\ref{sec:invariantcurves-vf}. Finally, in Section~\ref{sec:LRRS-LR} we present two generalizations of Écalle-Hakim theorem obtained in collaboration with J. Raissy, J. Ribón, R. Rosas, F. Sanz-Sánchez and L. Vivas.

\section{Invariant curves for vector fields in $(\C^2,0)$}\label{sec:invariantcurves-vf}\label{sec:vectorfields}

In this section we recall some classical results about the reduction of singularities and the existence of invariant curves for formal vector fields in $(\C^2,0)$. More details for the contents of this section can be found for instance in \cite{Cam-S2}, \cite{Can-C-D} or \cite{Ily-Y}.

A formal vector field $X$ in $(\mathbb{C}^2,0)$ can be written, in a unique way up to multiplication by a unit, as 
$$X=A(x,y)\parcial x+B(x,y)\parcial y=f\left[a(x,y)\parcial x+b(x,y)\parcial y\right],$$ 
where $f,A,B,a,b\in\C[[x,y]]$ and $a$ and $b$ have no common factor. The \emph{order} of $X$ is defined as $\ord X=\min\{\nu_0(A),\nu_0(B)\}$. The vector field 
$$\sat X=a(x,y)\parcial x+b(x,y)\parcial y$$ 
is called the \emph{saturation} of $X$, and we say that $X$ is saturated if $f$ is a unit. We say that $X$ is \emph{singular} if $\ord X\ge1$, and that $X$ is \emph{strictly singular} if $\sat X$ is singular. If $X$ is not saturated, the \emph{singular locus} of $X$ is $\sing X=\sqrt{(f)}$; if $X$ is saturated, $\sing X=\{0\}$ if $X$ is singular and $\sing X=\emptyset$ otherwise. 
An irreducible formal curve $\Gamma=(g)$ in $(\mathbb{C}^2,0)$ is said to be a (formal) \emph{invariant curve} of $X$ if $X(g)\in\Gamma$. Equivalently, if $\gamma(t)=\left(\gamma_1(t),\gamma_2(t)\right)$ is a parametrization of $\Gamma$, we have that
$$\left(A(\gamma(t)),B(\gamma(t))\right)=h(t)\gamma'(t)$$
for some $h\in\C[[t]]$, i.e. 
$$A(\gamma_1(t),\gamma_2(t))\gamma_2'(t)=B(\gamma_1(t),\gamma_2(t))\gamma_1'(t).$$ 
Observe that the components of the singular locus of $X$ are invariant curves, and if $X$ is non-singular then its formal integral curve through the origin is its only invariant curve. 

Assume that $X$ is singular, let $\pi$ be the blow-up centered at 0 and denote by $E=\pi^{-1}(0)$ the exceptional divisor and by $\tilde X_p$ the transform of $X$ by $\pi$ at $p$ for any $p\in E$. We have that a formal curve $\Gamma$ is invariant for $X$ if and only if its strict transform $\tilde\Gamma$ by $\pi$ is invariant for $\tilde X_q$, where $\{q\}=\tilde\Gamma\cap E$. Denote $k+1=\ord X$, write 
$$A(x,y)=A_{k+1}(x,y)+A_{k+2}(x,y)+\cdots, \quad B(x,y)=B_{k+1}(x,y)+B_{k+2}(x,y)+\cdots,$$
where $A_j,B_j$ are homogeneous of degree $j$, and define
\begin{equation}\label{eq:tangentcone}
P_X(x,y)=xB_{k+1}(x,y)-yA_{k+1}(x,y).
\end{equation} 
If $p=[x:y]\in E$ corresponds to a zero of $P_X$, then $\ord \tilde X_p\ge\ord X$. When $P_X\equiv0$, we say that $X$ is \emph{dicritical}, otherwise it is \emph{non-dicritical}. If $X$ is dicritical, $E$ is invariant for $\tilde X_p$ but not for $\sat \tilde X_p$ for every $p\in E$ and if $p\in E$ is a generic point (corresponding to a direction $[x:y]$ such that $\left(A_{k+1}(x,y),B_{k+1}(x,y)\right)\neq0$) then $\sat \tilde X_p$ is non-singular, so $X$ has infinitely many formal invariant curves. If $X$ is non-dicritical, $E$ is invariant for $\sat\tilde X_p$ for every $p\in E$ and every zero of $P_X$ is either the tangent of an invariant curve contained in $\sing X$ or corresponds, after a blow-up, to a point $p$ where $\tilde X_p$ is strictly singular.

\subsection{Reduction of singularities}
We say that a saturated singular vector field $X$ in $(\mathbb{C}^2,0)$ has an \emph{elementary} singularity if its linear part
$$\left(\begin{array}{cc}\displaystyle \Parcial a x (0) & \displaystyle\Parcial a y(0) \\[7pt]
\displaystyle\Parcial b x (0) &\displaystyle\Parcial b y(0)\\ \end{array}\right)$$
is non nilpotent, and we say that it has a \emph{simple} singularity if the eigenvalues $\lambda,\mu$ of the linear part satisfy $\lambda\neq0$ and $\mu/\lambda\not\in\N^*\cup 1/\N$. We say that a simple singularity is \emph{non-degenerate} if $\lambda\mu\neq0$, otherwise we say that the singularity is a \emph{saddle-node}.

If $X$ has a simple singularity, then it is non-dicritical and after a blow-up we find exactly two points in the exceptional divisor where the transform of $X$ is strictly singular: both are elementary and at least one of them is simple.

\begin{remark}
Note that our definition of simple singularities is not the classical one: usually the condition required to have a simple singularity is that $\lambda\neq0$ and $\mu/\lambda\not\in \Q_{>0}$. The discrepancy between the two definitions is only technical and does not change the nature of the results. 
\end{remark}

\begin{theorem}[Seidenberg \cite{Sei}]\label{th:Seidenberg}
Given a formal vector field $X$ in $(\C^2,0)$, there is a map $\sigma:M\to\C^2$, composition of a finite number of blow-ups, such that for any $p\in\sigma^{-1}(0)$ we can find coordinates $(x,y)$ such that the germ of $\sigma^*X$ at $p$ is written as
$$(\sigma^*X)_p=x^my^n\tilde X$$
where $m\ge1$, $n\ge0$, $\ord (\sigma^*X)_p\ge\ord X$ and $\tilde X$ is either non-singular or has a simple singularity. Any such map is called a \emph{reduction of singularities} of $X$.
\end{theorem}

The proof of Seidenberg theorem can also be found in \cite{Mat-M}, \cite{Cam-S2} or \cite{Can-C-D}, for instance. Note that in those references the reduction of singularities is carried out for saturated vector fields; the result stated here can be easily obtained from that case up to some additional blow-ups to handle the singular locus.

\subsection{Invariant curves at simple singularities}

\begin{theorem}[Briot, Bouquet \cite{Bri-B}]\label{th:Briot-Bouquet}
If $X$ has a simple singularity, then $X$ has exactly two non-singular formal invariant curves $\Gamma_1,\Gamma_2$ which are moreover tangent to the eigendirections of the linear part of $X$. Moreover, if $X$ is analytic and the eigenvalue corresponding to the tangent of $\Gamma_i$ is non-zero then $\Gamma_i$ is analytic.
\end{theorem}

We say that $\Gamma_i$ is \emph{strong} if it is tangent to the eigenspace of a non-zero eigenvalue, otherwise we say that it is \emph{weak}. 

\begin{proof}
Let $\lambda, \mu$ be the eigenvalues of the linear part of $X$. Since $\lambda\ne\mu$, up to a linear change of coordinates we can write
$$X=\left(\lambda x+a(x,y)\right)\parcial x+\left(\mu y+b(x,y)\right)\parcial y,$$
with $\nu_0(a),\nu_0(b)\ge2$. We look for a formal invariant curve of $X$ of the form
$$\Gamma_1=\biggl(y-\sum_{n\ge1} c_nx^n \biggr).$$
$\Gamma_1$ being invariant for $X$ means that
\begin{equation}\label{eq:invariance}
\biggl(\lambda x+a\Bigl(x,\sum c_nx^n\Bigr)\biggr)\sum nc_nx^{n-1}=\mu\sum c_nx^n+ b\Bigl(x,\sum c_nx^n\Bigr).
\end{equation}
Identifying the coefficients of $x^n$ in this equation we obtain $(\lambda-\mu)c_1=0$ and
$$(\lambda n-\mu)c_n=P_n(c_1,...,c_{n-1}),$$
for every $n\ge2$, where $P_n$ is a polynomial that depends only on $a$ and $b$. Since $\lambda n-\mu\neq0$ for all $n\ge1$, because $X$ has a simple singularity, there exists a unique formal solution of equation~\eqref{eq:invariance} with $c_1=0$. Analogously, if we look for a formal invariant curve of $X$ of the form
$$\Gamma_2=\biggl(x-\sum_{n\ge1} d_ny^n \biggr),$$
we obtain $(\mu-\lambda)d_1=0$ and 
$$(\mu n-\lambda)d_n=Q_n(d_1,...,d_{n-1}),$$
for every $n\ge2$ and for some polynomial $Q_n$ that only depends on $a$ and $b$. Since $\mu n-\lambda\neq0$ for all $n\ge1$, we find a unique formal solution with $d_1=0$.

Assume now that $X$ is analytic and $\lambda\neq0$ and let us show that $\Gamma_1$ is analytic. Let $\pi$ be the blow-up at 0 and $\{p\}=\pi^{-1}(0)\cap\tilde\Gamma_1$, where $\tilde\Gamma_1$ is the strict transform of $\Gamma_1$. In coordinates $(x,v)$ at $p$ we have that $\tilde\Gamma_1=\left(v-\sum_{n\ge1} c_{n+1}x^n\right)$ and
$$\left(\pi^*X\right)_p=x\left(\lambda+\tilde a(x,v)\right)\parcial x+\left((\mu-\lambda)v+\tilde b(x,v)\right)\parcial v,$$
where $\tilde a(x,v)=x^{-1}a(x,xv)$ and $\tilde b(x,v)=x^{-1}(b(x,xv)-va(x,xv))$ so $\nu_0(\tilde a),\nu_0(\tilde b)\ge1$ and $\Parcial{\tilde b}{v}(0)=0$. In these coordinates, the invariance equation for $\tilde\Gamma_1$ is
$$\biggl(\lambda +\tilde a\Bigl(x,\sum c_{n+1}x^n\Bigr)\biggr)\sum nc_{n+1}x^n=(\mu-\lambda)\sum c_{n+1}x^n+ \tilde b\Bigl(x,\sum c_{n+1}x^n\Bigr).$$
Since $\lambda\neq0$, we have
$$\sum nc_{n+1}x^n=\frac{\mu-\lambda}\lambda \sum c_{n+1}x^n+h\Bigl(x,\sum c_{n+1}x^n\Bigr)$$
where $h(x,v)=(\mu/\lambda-1)v\left[(1+\tilde a(x,v)/\lambda)^{-1}-1\right]+\tilde b(x,v)(\lambda+\tilde a(x,v))^{-1}$ is analytic. If we write 
$$h(x,v)=\sum_{i+j\ge1\atop (i,j)\neq(0,1)}h_{i,j}x^iv^j$$ 
we have that
$$(n+1-\mu/\lambda)c_{n+1}=\sum_{i+j\ge1, \, (i,j)\neq(0,1)\atop i+n_1+\cdots+n_j=n} h_{i,j}c_{n_1+1}...c_{n_j+1}$$
for every $n\ge1$. Set $\gamma=\inf\{|n+1-\mu/\lambda|:n\in\N^*\}>0$. Since $\gamma\neq0$, by the implicit function theorem the equation 
$$\gamma v=\sum_{i+j\ge1\atop (i,j)\neq(0,1)}|h_{i,j}|x^iv^j$$
has a unique analytic solution of the form $v=\sum_{n\ge1}v_nx^n$, where
$$\gamma v_n=\sum_{i+j\ge1, \, (i,j)\neq(0,1)\atop i+n_1+\cdots+n_j=n} |h_{i,j}|v_{n_1}...v_{n_j}.$$
By induction we have that $|c_{n+1}|\le v_n$, so $\tilde\Gamma_1$ is analytic and hence so is $\Gamma_1$. Analogously, if $X$ is analytic and $\mu\neq0$ we have that $\Gamma_2$ is analytic.
\end{proof}

\begin{remark}
The invariant curves provided by Theorem~\ref{th:Briot-Bouquet} are not the only invariant curves of $X$ in general. For instance, the vector field 
$$X=x\parcial x+\frac pq y\parcial y$$
has $(y^q-cx^p)$ as invariant curve for any $c\in\C$. 
\end{remark}

\begin{remark}
Weak invariant curves are divergent in general even when $X$ is analytic: for instance, the weak invariant curve for Euler's vector field
$$X=x^2\partial_x+(y-x)\partial_y$$
is $\Gamma=\left(y-\sum_{n\ge0} n!x^{n+1}\right)$, which is not analytic. However, as was proved in \cite{Mar-R}, weak invariant curves are always \emph{summable} in the sense of Borel-Laplace summability \cite{Ram} (see also \cite{Bal}), which implies that we can find an analytic function defined in a sector of wide opening which is invariant for $X$ and has the formal invariant curve as Gevrey asymptotic expansion. 
\end{remark}

\begin{remark}\label{rem:Poincare-Dulac}
What happens in the case where $X$ has an elementary singularity which is not simple? Assume, with the notations in the proof of Theorem~\ref{th:Briot-Bouquet}, that $\mu=k\lambda$, with $k\in\N^*$ and $\lambda\neq0$. From the proof of Theorem~\ref{th:Briot-Bouquet} we see that there is a unique invariant curve of the form $\Gamma=\bigr(x-\sum_{n\ge2} d_n y^n\bigl)$, analytic if $X$ is analytic (in case $k=1$ we have to make a slight change in the proof and include a term of the form $\beta x$ in the second component of $X$, since the linear part of $X$ might be non-diagonalizable). However equation~\eqref{eq:invariance} does not have a solution if $P_k(c_1,...,c_{k-1})\neq0$ (and has an infinite number of them otherwise). Let us show that we can find anyhow a formal transeries different from $\Gamma$ which is invariant for $X$ (see \cite[Chapitre III]{Huk-K-M}). First, up to a formal change of coordinates $(x,y)\mapsto(t,u)=(x-\sum d_ny^n,y)$ (which is analytic if $X$ is analytic), we can assume that $\Gamma=(t)$, so $X$ can be written in the coordinates $(t,u)$ as
$$X=t\left(\lambda+\tilde a(t,u)\right)\parcial t+\left(k\lambda u+\beta t+\tilde b(t,u)\right)\parcial u,$$
with $\nu_0(\tilde a)\ge1$, $\nu_0(\tilde b)\ge2$ and $\beta=0$ if $k\neq1$. Then, finding a formal invariant curve of $X$ is equivalent to finding a formal invariant curve of 
$$\tilde X=t\parcial t+\left(k u+\frac{\beta}{\lambda}t+ \hat b(t,u)\right)\parcial u$$
where $\hat b(t,u)=\tilde b(t,u)(\lambda+\tilde a(t,u))^{-1}+(ku+\beta t/\lambda)\left((1+\tilde a(t,u)/\lambda)^{-1}-1\right)$. By Poincaré-Dulac theorem we know that $\tilde X$ is formally conjugate by a change of coordinates of the form $(t,u)\mapsto (z,w)$ with $t=z$, $u=w+a_{1,0}z+\sum_{i+j\ge2}a_{i,j}z^iw^j$, where $a_{1,0}=0$ if $\beta=0$, to 
$$z\parcial z+\left(k w+az^k\right)\parcial w$$
for some $a\in\C$; moreover, since $\lambda,\mu$ are in the Poincaré domain the conjugation is analytic if $X$ is analytic (see \cite{Dul}). If $a=0$, we find infinitely many invariant curves of the form $\left(w-cz^k\right)$ for $c\in\C$. If $a\neq0$, we have that $w=a z^k\log z+cz^k$ is invariant for every $c\in\C$, so in the coordinates $(t,u)$ we find infinitely many invariant power-logarithmic series (also called Dulac series) of $\tilde X$ of the form
$$\Bigl(u-\sum_{i+j\ge1} c_{i,j}t^i (at^k\log t+ct^k)^j\Bigr),$$
which are analytic if $\tilde X$ is analytic.
\end{remark}

\subsection{Existence of invariant curves}

\begin{theorem}[Camacho, Sad \cite{Cam-S}]\label{th:Camacho-Sad}
Let $X$ be a singular saturated formal vector field in $(\C^2,0)$. There exist a sequence $\sigma$ of blow-ups and a point $p\in\sigma^{-1}(0)$ satisfying the following: there are some coordinates $(x,y)$ at $p$ such that $\sigma^{-1}(0)=(x=0)$ and
\begin{equation}\label{eq:Camacho-Sad}
(\sigma^*X)_p=x^m\left[(\lambda x+a(x,y))\parcial x+(\mu y+b(x,y))\parcial y\right]
\end{equation}
where $\lambda\neq0$, $\mu/\lambda\not\in\N^*\cup 1/\N$, $\nu_0(a),\nu_0(b)\ge2$ and $m+1\ge \ord X$.
\end{theorem}

\begin{remark}
Theorem~\ref{th:Camacho-Sad} guarantees the existence of a formal invariant curve of any formal vector field $X$, which is analytic in case $X$ is analytic. If $X$ is non-singular or if $X$ is not saturated, the result is clear; otherwise, Theorem~\ref{th:Camacho-Sad} guarantees the existence of a formal invariant curve $\tilde\Gamma$ of $(\sigma^*X)_p$ which is transverse to $\sigma^{-1}(0)$ and analytic if $X$ is analytic: if $\sat\bigl((\sigma^*X)_p\bigr)$ is non-singular then $\tilde\Gamma$ is the integral curve of $\sat\bigl((\sigma^*X)_p\bigr)$ at $p$, otherwise $\sat\bigl((\sigma^*X)_p\bigr)$ has a simple singularity, and the existence of $\tilde\Gamma$ follows from Theorem~\ref{th:Briot-Bouquet}. Since $\tilde\Gamma\neq\sigma^{-1}(0)$, the blow-down of $\tilde\Gamma$ is a formal invariant curve of $X$.
\end{remark}

\begin{proof}[Proof of Theorem~\ref{th:Camacho-Sad}]
The proof uses the Camacho-Sad index of a vector field relative to a formal invariant curve. Let $Y$ be a singular formal vector field in $(\C^2,0)$ and assume that $\sat Y$ has a non-singular invariant curve $\Gamma$. Up to a formal change of coordinates, we can assume that $\Gamma=(x)$, so we can write
$$\sat Y=x a(x,y)\parcial x+b(x,y)\parcial y.$$
The Camacho-Sad index of $Y$ relative to $\Gamma$ is defined as
$$\cs(Y,\Gamma)=\res_0\left(\frac{a(0,y)}{b(0,y)}\right).$$

This index satisfies the following properties:
\begin{enumerate}[(i)]
\item If $Y$ is non-dicritical and $\pi$ is the blow-up at $0$, then
$$\sum_{p\in\pi^{-1}(0)} \cs\left((\pi^*Y)_p,\pi^{-1}(0)\right)=-1.$$
\item If $\pi$ is the blow-up at $0$, $\tilde\Gamma$ is the strict transform of $\Gamma$ and $\{p\}=\tilde\Gamma\cap \pi^{-1}(0)$, then
$$\cs\bigl((\pi^*Y)_p,\tilde\Gamma\bigr)=\cs(Y,\Gamma)-1.$$
\item Assume that $\sat Y$ has a simple singularity. If it is non-degenerate and $\Gamma_1, \Gamma_2$ are the strong formal invariant curves then
$$\cs(Y,\Gamma_1)\cs(Y,\Gamma_2)=1;$$
if it is a saddle-node and $\Gamma$ is the strong formal invariant curve then 
$$\cs(Y,\Gamma)=0.$$
\end{enumerate}

Instead of the original proof of Camacho and Sad, which is based on a combinatorial argument on the reduction of singularities, we present here an alternative constructive proof due to Cano \cite{Can} (for another constructive proof, that makes use of the Newton-Puiseux polygon, see \cite{Can2}).

Consider a singular formal vector field $Y$ in $(\C^2,p)$ and let $D$ be a normal crossing divisor. We say that the pair $(Y,D)$ satisfies property $(\star)$ if one of the following conditions holds:
\begin{enumerate}
\item[$(\star_1)$] $p$ belongs to a unique component $S$ of $D$ which is invariant for $\sat Y$ and
$$ \cs(Y,S)\not\in\Q_{\ge 0}.$$
\item[$(\star_2)$] $p$ belongs to two components $S_1$ and $S_2$ of $D$ which are invariant for $\sat Y$ and there exists a real number $a>0$ such that
$$\cs(Y,S_1)\in\Q_{\le -a}\quad\text{and}\quad \cs(Y,S_2)\not\in\Q_{\ge -1/a}.$$
\end{enumerate}

Using property (i) of the Camacho-Sad index, it is easy to see that if $Y$ is a non-dicritical singular vector field in $(\C^2,0)$ and $\pi$ is the blow-up at 0 then there exists a point $p\in\pi^{-1}(0)$ such that $((\pi^*Y)_p,\pi^{-1}(0))$ satisfies property $(\star_1)$. Moreover, if $(Y,D)$ satisfies property $(\star)$ at $p$, $Y$ is non-dicritical and $\pi$ is the blow-up at $p$, then it can be shown using properties (i) and (ii) that there exists a point $q\in\pi^{-1}(p)$ such that $((\pi^*Y)_q,\pi^{-1}(D))$ satisfies property $(\star)$. 

To prove Theorem~\ref{th:Camacho-Sad} we proceed in the following way. Let $X$ be a singular saturated formal vector field and let $\pi_1$ be the blow-up at 0; if $X$ is dicritical, then we are done considering a generic point in the exceptional divisor, otherwise we find a point $p_1\in\pi_1^{-1}(0)$ such that 
$$(X_1,D_1)=\bigl((\pi_1^*X)_{p_1},\pi_1^{-1}(0)\bigr)$$ 
satisfies property $(\star_1)$. Let $\pi_2$ be the blow-up at $p_1$: if $X_1$ is dicritical then we are done as in the previous step, otherwise we find a point $p_2\in\pi_2^{-1}(p_1)$ such that 
$$(X_2,D_2)=\left((\pi_2^*X_1)_{p_2},\pi_2^{-1}(D_1)\right)$$ satisfies property $(\star)$. We continue this process: if at some step we find a dicritical vector field we are done, otherwise by Theorem~\ref{th:Seidenberg} we find a point $p_n$ such that $(X_n,D_n)$ satisfies property $(\star)$ and $X_n$ has a simple singularity at $p_n$. But property (iii) guarantees that in this case $(X_n,D_n)$ cannot satisfy $(\star_2)$, so $(\star_1)$ holds and moreover the singularity is either non-degenerate or a saddle-node with strong separatrix transverse to $D_n$, so we can find coordinates $(x,y)$ at $p_n$ such that $X_n$ is written as in~\eqref{eq:Camacho-Sad}.
\end{proof}

\section{Parabolic curves for biholomorphisms in $\C^2$: Écalle-Hakim and Abate theorems}\label{sec:HakimAbate}
In this section and the next one we will be concerned with the problem of finding stable manifolds for tangent to the identity biholomorphisms in $(\C^2,0)$. Before introducing the definitions and stating the first results, let us recall the description of the local dynamics of tangent to the identity biholomorphisms in $(\C,0)$ given by Leau-Fatou flower theorem \cite{Lea,Fat} (see also \cite{Mil} or \cite{Lor}).
Given $k\ge1$, $\eta\in(0,\pi/2)$ and $\delta>0$, we denote by $\petal$ the set
$$\petal=\hspace{-5pt}\mathop{\bigcup}\limits_{-\eta\le \theta\le\eta}\hspace{-5pt}\left\{x\in\C:|x^k-\delta e^{i\theta}|<\delta\right\}.$$
It has $k$ connected components that are fitted in sectors of opening $(\pi+2\eta)/k$ bisected by the half-lines $\{x\in\C:x^k\in\R^+\}$. Consider now a tangent to the identity biholomorphism $f$ in $(\C,0)$ with $f\neq\id$. We can write $f(x)=x+ax^{k+1}+O(x^{k+2})$ for some $a\neq0$ and some $k\ge1$, and up to a change of variable of the form $x\mapsto (-a)^{1/k}x$ we can assume that $a=-1$. 
\begin{theorem}[Leau \cite{Lea}, Fatou \cite{Fat}]\label{th:LeauFatou}
Consider a tangent to the identity biholomorphism $f(x)=x-x^{k+1}+O(x^{k+2})$ in $(\C,0)$. Given $\eta\in(0,\pi/2)$, there 
exists $\delta_0>0$ such that for every $\delta\le \delta_0$ and every component $\Delta$ of $\petal$ we have
$$f(\Delta)\subset\Delta\quad\text{ and }\quad f^n(p)\to0 \,\,\text{ for every } p\in\Delta.$$
Moreover, $kn(f^n(p))^k\to 1$ for every $p\in\Delta$, so every orbit in $\Delta$ converges tangentially to the half-line bisecting $\Delta$. 
\end{theorem}

\begin{remark} 
Since $f^{-1}(x)=x+x^{k+1}+O(x^{k+2})$, Theorem~\ref{th:LeauFatou} guarantees that given $\eta\in(0,\pi/2)$, there exists $\delta_1>0$ such that for every $\delta\le \delta_1$ and every component $\Omega$ of 
$$\Omega^k_{\eta,\delta}=\hspace{-10pt}\mathop{\bigcup}\limits_{\pi-\eta\le \theta\le \pi+\eta}\hspace{-5pt}\left\{x\in\C:|x^k-\delta e^{i\theta}|<\delta\right\}$$ we have
that $f^{-1}(\Omega)\subset\Omega$ and  $f^{-n}(p)\to0$ for every $p\in\Omega.$
Then the union of the components of $\petal$ (called attracting petals of $f$) and the components of $\Omega^k_{\eta,\delta}$ (the attracting petals of $f^{-1}$, also called repelling petals of $f$) is a punctured neighborhood of 0 (see figure~\ref{fig:flower}).
\begin{figure}[h]
\begin{centering}
\includegraphics[scale=.7]{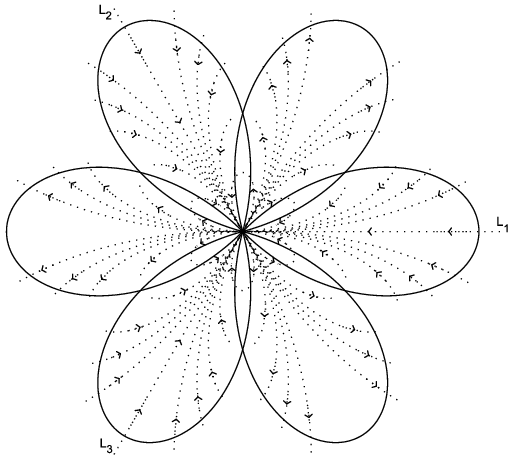}\caption{\small{Attracting and repelling petals for $f(z)=z-z^4+O(z^5)$. The attracting petals are bisected by the half-lines $L_1$, $L_2$ and $L_3$.}}
\label{fig:flower}
\end{centering}
\end{figure}
\end{remark}
  
Consider now a tangent to the identity biholomorphism $F$ in $(\C^2,0)$. 
A \emph{parabolic manifold} of dimension $r$ for $F$ is a connected submanifold $\Omega\subset\C^2$ of dimension $r$ such that $0\in\partial\Omega$, $F(\Omega)\subset\Omega$ and $F^n(p)\to0$ for every $p\in\Omega$. When $r=1$, we say that $\Omega$ is a \emph{parabolic curve} and when $r=2$ we say that it is a \emph{parabolic domain}.

Note that if $F$ has an analytic invariant curve $\Gamma$ which is not contained in the set of fixed points of $F$ (i.e. the set of points $(x,y)$ such that $F(x,y)=(x,y)$) then the existence of parabolic curves is immediate, since the restriction of $F$ to $\Gamma$ is a tangent to the identity biholomorphism in $(\C,0)$ which is not the identity, and we can apply Theorem~\ref{th:LeauFatou}. However, this is not the case in general (see for instance \cite[Section 3]{Hak}).

If we write  
\begin{equation}\label{eq:difeopq}
F(x,y)=\left(x+p_{k+1}(x,y)+p_{k+2}(x,y)+\cdots\!, y+q_{k+1}(x,y)+q_{k+2}(x,y)+\cdots\!\right)
\end{equation}
where $p_j,q_j$ are homogeneous polynomials of degree $j$ and $(p_{k+1},q_{k+1})\neq0$, we say that $k+1$ is the \emph{order} of $F$. A \emph{characteristic direction} of $F$ is a direction $[v]\in\P_\C$ such that 
$$(p_{k+1}(v),q_{k+1}(v))=\lambda v$$ 
for some $\lambda\in\C$. The dynamical interest of characteristic directions is the following: if $F$ has an invariant curve, then its tangent is a characteristic direction; moreover, if an orbit of $F$ converges tangentially to a direction $[v]$, then $[v]$ is a characteristic direction (see \cite[Proposition 2.3]{Hak}).

If $[v]$ is a characteristic direction and $\lambda\neq0$, we say that $[v]$ is \emph{non-degenerate}.

\begin{theorem}[Écalle \cite{Eca}, Hakim \cite{Hak}]\label{th:Ecalle-Hakim}
If $[v]$ is a non-degenerate characteristic direction of $F$ then there are parabolic curves for $F$ where the orbits converge tangentially to $[v]$.
\end{theorem}

Although we are only dealing with the two-dimensional case, let us mention that Theorem~\ref{th:Ecalle-Hakim} holds in any dimension.

To relate Theorem~\ref{th:Ecalle-Hakim} to the results in Section~\ref{sec:vectorfields}, let us explain briefly the relation between tangent to the identity biholomorphisms and formal vector fields given by the time-1 flow (see \cite{Lor} or \cite{Ily-Y} for more details).  

Let $F$ be a tangent to the identity biholomorphism in $(\C^2,0)$ of order $k+1$. There exists a unique formal vector field $X$ in $\Cd2$ of order at least two whose time-1 flow is $F$, namely
$$F(x,y)=\left(\sum_{n=0}^\infty \frac{X^n(x)}{n!},\sum_{n=0}^\infty\frac{X^n(y)}{n!}\right),$$
where $X^n$ denotes the $n$-th iteration of $X$ as a derivation, i.e. $X^0(f)=f$ and $X^{j+1}(f)=X(X^j(f))$.
This vector field is called the \emph{infinitesimal generator} of $F$ and denoted $\log F$. 

Is is easy to see that the order of $\log F$ coincides with the order of $F$, and if $F$ is written as in \eqref{eq:difeopq} then the jet of order $k+1$ of $\log F$ is
$$p_{k+1}(x,y)\parcial x+q_{k+1}(x,y)\parcial y.$$
Therefore, the polynomial $xq_{k+1}(x,y)-yp_{k+1}(x,y)$, whose zeros are the characteristic directions of $F$, coincides with the polynomial $P_{\log F}(x,y)$ defined in \eqref{eq:tangentcone}. Moreover, the set of fixed points of $F$ coincides with the singular locus of $\log F$, which is therefore analytic. 

An irreducible formal curve $\Gamma=(g)$ is \emph{invariant} for $F$ if $g\circ F\in \Gamma$. Equivalently, if $\gamma(t)=\left(\gamma_1(t),\gamma_2(t)\right)$ is a parametrization of $\Gamma$, we have that
$$F(\gamma(t))=\gamma(\theta(t))$$
for some $\theta\in\C[[t]]$. A formal curve $\Gamma$ is invariant for $F$ if and only if it is invariant for $\log F$ (see \cite[Proposition 3]{Rib}). 

Let $\pi:M\to\C^2$ be the blow-up at 0 and let $p\in \pi^{-1}(0)$ be the point corresponding to a characteristic direction of $F$. There exists a unique biholomorphism $\tilde F_p$ in $(M,p)$, called the \emph{transform} of $F$ by $\pi$ at $p$, such that $\pi\circ \tilde F_p=F\circ\pi$. This biholomorphism $\tilde F_p$ is tangent to the identity, its order is greater than or equal to the order of $F$ and moreover its infinitesimal generator is the transform of $\log F$ by $\pi$ at $p$ (see \cite{Rib} or \cite{Bro-C-L}).

Consider now a tangent to the identity biholomorphism $F$ in $(\C^2,0)$ with a non-degenerate characteristic direction $[v]$ and set $X=\log F$. Up to a linear change of coordinates we can assume that $[v]=[1:0]$, so with the notations above we have that 
$$p_{k+1}(1,0)=\lambda\neq0\quad\text{and}\quad q_{k+1}(1,0)=0.$$
If $\pi$ is the blow-up at 0 and $p\in\pi^{-1}(0)$ is the point that corresponds to $[v]$ we can find coordinates $(x,y)$ at $p$ such that $\pi^{-1}(0)=(x=0)$ and
\begin{equation}\label{eq:vfafterfblowup}
(\pi^*X)_p=x^k\left[x(\lambda+a(x,y))\parcial x+(\mu y+\beta x+b(x,y))\parcial y\right]
\end{equation}
with $\lambda\neq0$, $\mu, \beta \in\C$, $\nu_0(a)\ge1$ and $\nu_0(b)\ge2$, so $\sat\left((\pi^*X)_p\right)$ is either non-singular (and hence has a formal invariant curve transverse to $(x=0)$) or has an elementary singularity. Observe that the invariant curves of $\sat\left((\pi^*X)_p\right)$ are invariant for $(\pi^*X)_p$ and hence for $\tilde F_p$, since $\log(\tilde F_p)=(\pi^*X)_p$. Then, as we showed in Section~\ref{sec:invariantcurves-vf}, we have two possibilities:
\begin{enumerate}[a)]
\item either $\tilde F_p$ has a formal invariant curve of the form 
$$\Gamma=\left(y-\sum_{n\ge1}c_nx^n\right)$$
(this case always holds if $\mu/\lambda\not\in\N^*$)
\item or $\tilde F_p$ has an invariant Dulac series (an infinite number of them, actually) of the form
$$\Gamma=\left(y-\sum_{i+j\ge1}c_{i,j}x^i(x\log x)^j\right).$$
\end{enumerate}
As we will see in the proof, the parabolic curves provided by Écalle and Hakim theorem are analytic realizations of these invariant curves or Dulac series. In this sense, their result can be seen like a discrete version of Theorem~\ref{th:Briot-Bouquet}.

\begin{proof}[Proof of Theorem~\ref{th:Ecalle-Hakim}]
The strategy used by Écalle in \cite{Eca} is to prove the summable and resurgent character of the formal invariant curve or Dulac series $\Gamma$. We present here the strategy followed by Hakim, which is based on more standard arguments. 
As above, up to a change of coordinates we assume that $[v]=[1:0]$. Let $\pi$ be the blow-up at 0, let $p\in\pi^{-1}(0)$ be the point that corresponds to $[v]$ and let $\tilde F_p$ be the transform of $F$ at $p$. We can find coordinates $(x,y)$ at $p$ such that $\pi^{-1}(0)=(x=0)$ and $(\pi^*X)_p$ is written as in \eqref{eq:vfafterfblowup}, and up to a linear change of coordinates we can assume that $\lambda=-1$. Then, since $\log(\tilde F_p)=(\pi^*X)_p$ we have that $\tilde F_p=(\tilde F_1,\tilde F_2)$ has the form
\begin{align*}
\tilde F_1(x,y)&=x-x^{k+1}+O(x^{k+2},x^{k+1}y)\\
\tilde F_2(x,y)&=y+\alpha x^ky+O(x^ky^2,x^{k+1}y)+x^{k+1}\varphi(x)
\end{align*}
with $\alpha\in\C$ and $\varphi(x)\in\C\{x\}$. As we have seen, there exists a formal invariant curve or Dulac series for $\tilde F_p$ of the form $\Gamma=\left(y-\gamma(x)\right)$, where
$$\gamma(x)=\sum_{n\ge1}c_nx^n \quad\text{or}\quad \gamma(x)=\sum_{i+j\ge1}c_{i,j}x^i(x\log x)^j.$$
We will look for the parabolic curves as analytic realizations of $\Gamma$, so they will be given by the graph of an analytic function defined in an appropriate domain. This analytic function will be found, using Banach fixed point theorem, as the fixed point of some operator. Note first that the parabolic curve does not need to be unique: for instance, if $\tilde F_q$ is the time-1 flow of $x\left(-x\parcial x+\alpha y\parcial y\right)$ and $\alpha<0$ then $\{(x,cx^{-\alpha}):x\in\Delta^1_{\eta,\delta}\}$ is a parabolic curve of $\tilde F_p$ for every $c\in\C$ and $\eta\in(0,\pi/2)$ and for $\delta$ small enough. We need then to impose a high order of contact between the parabolic curve and the formal invariant curve or Dulac series, to be able to use 
Banach fixed point theorem. We start fixing $m\ge2$ such that $m>1-\re\alpha$ and performing an initial change of coordinates of the form
$$y\mapsto y-J_{m-1}\gamma(x),$$
where $J_{m-1}\gamma(x)$ is the jet of order $m-1$ of $\gamma$, so in the new coordinates the order of $\gamma$ is at least $m$ and $\tilde F_p$ is written
\begin{align*}
\tilde F_1(x,y)&=x-x^{k+1}+O(x^{k+2}\log x, x^{k+1}y)\\
\tilde F_2(x,y)&=y+\alpha x^ky+O(x^ky^2,x^{k+1}\log x\,y)+x^{k+m}\varphi_m(x,x\log x)
\end{align*}
with $\varphi_m(x,y)\in\C\{x,y\}$. Let $\ell$ be one of the half-lines $\xi\R^+$, where $\xi^{k}=1$ (these are the bisecting lines of the attracting petals of the restriction of $F$ to $\Gamma$, in case $\Gamma$ is an analytic invariant curve). Fix $\eta\in(0,\pi/2)$ and for any $\delta>0$ denote by $\Delta_{\delta}^{\ell}$ the component of $\petal$ bisected by $\ell$. 
Consider the space
$$\mathcal{B}_{\delta}=\left\{u\in\mathcal{O}(\Delta_{\delta}^{\ell},\C): \sup\Bigl\{\frac{|u(x)|}{|x^{m-1}\log x|}:x\in\Delta_{\delta}^{\ell}\Bigr\}<\infty\right\},$$
which is a Banach space with the norm $\|u\|=\sup\Bigl\{\frac{|u(x)|}{|x^{m-1}\log x|}:x\in\Delta_{\delta}^{\ell}\Bigr\}$, and its closed subset
$$\mathcal{H}_{\delta}=\left\{u\in\mathcal{B}_\delta: \|u\|\le1, |u'(x)|\le |x^{m-2}\log x|\,\,\forall x\in\Delta_{\delta}^\ell\right\}.$$
Given $u\in\mathcal{B}_{\delta}$, we denote 
$$f_u(x)=\tilde F_1(x,u(x)).$$
Arguing as in the proof of Leau-Fatou flower theorem, we have that if $u\in\mathcal{B}_{\delta}$ and $\delta$ is small enough then
$$f_u\left(\Delta_{\delta}^{\ell}\right)\subset \Delta_{\delta}^{\ell}\quad \text{ and } \quad  f_u^n(x)\to0  \,\,\text{ for any } x\in\Delta_\delta^\ell.$$
Therefore, if we find $u\in\mathcal{B}_{\delta}$ such that 
\begin{equation}\label{eq:invarianceparaboliccurve}
\tilde F_2(x,u(x))=u(\tilde F_1(x,u(x)))
\end{equation}
then $\{(x,u(x)):x\in\Delta_\delta^\ell\}$ will be a parabolic curve of $\tilde F_p$. Consider $u\in\mathcal{H}_\delta$ and denote $x_n=f_u(x_{n-1})$ for any $x_0\in \Delta_{\delta}^{\ell}$. If $\delta>0$ is sufficiently small, then it can be shown that the series
$$Tu(x_0)=x_0^{-\alpha}\sum_{n\ge0}\left[x_n^\alpha u(x_n)-x_{n+1}^\alpha\tilde F_2(x_n,u(x_n))\right]$$
is normally convergent and defines an element $Tu\in\mathcal{H}_\delta$; moreover $T:\mathcal{H}_\delta\to\mathcal{H}_\delta$ is a contraction. Observe that the solutions $u\in\mathcal{H}_\delta$ of \eqref{eq:invarianceparaboliccurve} are exactly the fixed points of $T$: the direct implication is immediate, for the converse simply write
$$Tu(x_0)=x_0^{-\alpha}\left[x_0^\alpha u(x_0)-x_1^\alpha \tilde F_2(x_0,u(x_0))+x_1^\alpha Tu(x_1)\right].$$
Then, by Banach fixed point theorem we find, for $\delta$ small enough, a solution $u\in \mathcal{H}_\delta$ of equation~\eqref{eq:invarianceparaboliccurve} and hence a parabolic curve for $\tilde F_p$, whose blow-down is a parabolic curve for $F$.
\end{proof}

\begin{remark}\label{rk:parabolicdomainsHakim}
Let $[v]$ be a non-degenerate characteristic direction of $F$ and assume, with the same notations of the proof of Theorem~\ref{th:Ecalle-Hakim}, that after a blow-up $\re\alpha<0$. Then Hakim proved in \cite{Hak2} (see also \cite{Ari-R}) that there exist parabolic domains $\Omega_1,...,\Omega_k$ of $F$ where every orbit converges tangentially to $[v]$, and moreover for any $j\in\{1,...,k\}$ there exists an injective holomorphic map $\varphi_j:\Omega_j\to\C^2$ conjugating $F$ with the map $(z,w)\mapsto(z+1,w)$. This result was later generalized by Vivas in \cite{Viv} (see also \cite{Lop-R2}) to the so-called \emph{fuchsian} characteristic directions.
\end{remark}

Although the condition of having a non-degenerate characteristic direction is generic, can we guarantee the existence of parabolic curves when it is not satisfied? Abate showed in \cite{Aba} that we always can, as long as $F$ has an isolated fixed point:

\begin{theorem}[Abate \cite{Aba}]\label{th:Abate}
Every tangent to the identity biholomorphism $F$ in $(\C^2,0)$ has either a curve of fixed points or parabolic curves where the orbits converge tangentially to a characteristic direction.
\end{theorem}

\begin{proof}
Abate's proof is an adaptation of Camacho-Sad theorem to the context of tangent to the identity biholomorphisms, to guarantee that if 0 is an isolated fixed point then after a finite number of blow-ups $\sigma$ we find a point $p\in\sigma^{-1}(0)$ and coordinates $(x,y)$ at $p$ such that the transform of $F$ at $p$ is written as
\begin{equation}\label{eq:Abate}
\tilde F(x,y)=\left(x+\lambda x^{m+1}+x^m A(x,y),y+\mu x^my+x^mB(x,y)\right),
\end{equation}
with $\lambda\neq0$, $\mu/\lambda\not\in\N^*\cup 1/\N$ and $\nu_0(A),\nu_0(B)\ge2$. Since $[v]=[1:0]$ is a non-degenerate characteristic direction for $\tilde F$, by Theorem~\ref{th:Ecalle-Hakim} we have that $\tilde F$ has parabolic curves where the orbits converge tangentially to $[v]$. The blow-down of these parabolic curves are parabolic curves for $F$.

An alternative approach to prove the theorem, given in \cite{Bro-C-L}, is to apply directly Camacho-Sad theorem to the infinitesimal generator $X$ of $F$. If 0 is an isolated fixed point of $F$, then $X$ is saturated so by Theorem~\ref{th:Camacho-Sad} there is a finite sequence of blow-ups $\sigma$, a point $p\in\sigma^{-1}(0)$ and coordinates $(x,y)$ at $p$ such that $(\sigma^*X)_p$ is written as in \eqref{eq:Camacho-Sad}. The time-1 flow of $(\sigma^*X)_p$, which is the transform of $F$ by $\sigma$ at $p$, can the be written in those coordinates as in \eqref{eq:Abate} and the existence of parabolic curves follows again by Theorem~\ref{th:Ecalle-Hakim}.
\end{proof}

\section{Parabolic domains for biholomorphisms in $\C^2$: two generalizations of Écalle-Hakim result}\label{sec:LRRS-LR}

Two natural questions arise from Theorem~\ref{th:Ecalle-Hakim}:
\begin{enumerate}[$\circ$]
\item The theorem shows that non-degenerate characteristic directions of $F$ support parabolic curves. Is this true for any characteristic direction of $F$?
\item It also shows that some formal invariant curves of $F$ (the ones that are strong invariant curves of $\log F$) support parabolic curves. Is this true for every formal invariant curve of $F$?
\end{enumerate}

Concerning the second question, the answer is essentially positive:

\begin{theorem}[López-Hernanz, Raissy, Ribón, Sanz \cite{Lop-R-R-S}]\label{th:LRRS}
Let $F$ be a tangent to the identity biholomorphism in $(\C^2,0)$ and let $\Gamma$ be a formal invariant curve of $F$ not contained in the set of fixed points of $F$. Then there exist parabolic manifolds $\Omega_1,\Omega_2,...,\Omega_r$ of $F$, where each $\Omega_i$ is either a parabolic curve or a parabolic domain, such that all the orbits in $\Omega_1\cup...\cup\Omega_r$ are asymptotic to $\Gamma$. 
\end{theorem}

Given a formal curve $\Gamma$, we say that an orbit $\{(x_n,y_n)\}$ of $F$ is \emph{asymptotic} to $\Gamma$ if the following condition holds: denoting by $\{q_n\}$ the sequence of infinitely near points of $\Gamma$, we have that if $\pi_1:M_1\to\C^2$ is the blow-up at $0$ then $\pi_1^{-1}(x_n,y_n)\to q_1$; if $\pi_2:M_2\to M_1$ is the blow-up at $q_1$, then $\pi_2^{-1}\circ\pi_1^{-1}(x_n,y_n)\to q_2$ and so on. In this case, $\Gamma$ is necessarily invariant for $F$.

\begin{proof}
The proof is based on the same strategy as Hakim's proof of Theorem~\ref{th:Ecalle-Hakim}. After a finite number of blow-ups centered at the infinitely near points of $\Gamma$ and polynomial changes of coordinates, we find coordinates $(x,y)$ such that the strict transform $\tilde\Gamma$ of $\Gamma$ is non-singular and transverse to $(x=0)$ and the transform $\tilde F=(\tilde F_1,\tilde F_2)$ of $F$ is written as
\begin{align*}
\tilde F_1(x,y)&=x-x^{k+p+1}+O(x^{2k+2p+1})\\
\tilde F_2(x,y)&=y\exp\left(x^kA(x)+O(x^{k+p+1})\right)+O(x^{k+1})
\end{align*}
where $k\ge1$, $p\ge0$ and $A(x)$ is a polynomial of degree at most $p$ with $A(0)\neq0$ if $p\ge1$. Observe that $\tilde F$ is close to the biholomorphism $G$ given by the time-1 flow of the vector field
$$Y=x^k\left[-x^{p+1}\parcial x+yA(x)\parcial y\right].$$
If $\{(x_n,y_n)\}$ is an orbit of $G$ converging to 0, then by Leau-Fatou flower theorem $\{x_n\}$ converges tangentially to a direction $\ell=\xi\R^+$, with $\xi^{k+p}=1$. Moreover, $Y$ has 
\begin{equation}\label{eq:defE}
yE(x)=y\exp\left(\displaystyle\int \frac{A(x)}{x^{p+1}}\,dx \right)
\end{equation}
as first integral, so the behavior of the orbits depends on the asymptotics of $E$ near the direction $\ell$. According to the dynamics of the toy model $G$, we fix a direction $\ell=\xi\R^+$ with $\xi^{k+p}=1$, and assume up to a linear change of coordinates that $\ell=\R^+$. We say that $\ell$ is a \emph{node direction} if $p\ge1$ and
$$\left(\re A_0,\re A_1,...,\re A_{p-1}\right)<0$$
in the lexicographic order, where $A(x)=A_0+A_1x+\cdots+A_px^p$; otherwise, we say that $\ell$ is a \emph{saddle direction}. Observe that, since we are only interested in the orbits that converge asymptotically to $\Gamma$, in case $p=0$ we always consider $\ell$ as a saddle direction: in case $\re A_0<0$ there exists a parabolic domain for each $\ell$ as we mentioned in Remark~\ref{rk:parabolicdomainsHakim}, but the orbits in this parabolic domain (except for the ones in the parabolic curve provided by Theorem~\ref{th:Ecalle-Hakim}) are not asymptotic to $\Gamma$. 

Up to a finite number of extra blow-ups following the infinitely near points of $\Gamma$, that do not change the coefficients $A_0,...,A_{p-1}$ and hence the saddle or node character of $\ell$, we assume that 
$$\re A_p>0.$$ 
We fix $m\ge p+2$ and consider the change of coordinates $(x,y)\mapsto (x,y-J_{p+m-1}\gamma(x))$, where $(x,\gamma(x))$ is a parametrization of $\tilde\Gamma$. In the new coordinates, the order of $\gamma$ is at least $p+m$ and $\tilde F$ is written
\begin{align*}
\tilde F_1(x,y)&=x-x^{k+p+1}+O(x^{2k+2p+1})\\
\tilde F_2(x,y)&=y\exp\left(x^kA(x)+O(x^{k+p+1})\right)+O(x^{k+p+m}).
\end{align*}

Denote by $r$ the first index $0\le r\le p$ such that $\re A_r\ne0$ (which is well defined since $\re A_p>0$). Given $d,e,\varepsilon>0$, denote
$$R_{d,e,\varepsilon}=\left\{x\in\C: |x|<\varepsilon, \re x>0, -d(\re x)^{r+1}<\im x<e(\re x)^{r+1}\right\}.$$

Assume first that $\ell$ is a saddle direction. In this case, we will find a parabolic curve of $\tilde F$ given by the graph of an analytic function defined in an appropriate domain. Consider the space 
$$\mathcal{B}_{d,e,\varepsilon}=\left\{u\in\mathcal{O}(R_{d,e,\varepsilon},\C): \sup\Bigl\{\frac{|u(x)|}{|x|^{m-1}}:x\in R_{d,e,\varepsilon}\Bigr\}<\infty\right\},$$
which is a Banach space with the norm $\|u\|=\sup\Bigl\{\frac{|u(x)|}{|x|^{m-1}}:x\in R_{d,e,\varepsilon}\Bigr\}$, and its closed subset
$$\mathcal{H}_{d,e,\varepsilon}=\left\{u\in\mathcal{B}_{d,e,\varepsilon}: \|u\|\le1, |u'(x)|\le |x|^{m-p-2}\,\,\forall x\in R_{d,e,\varepsilon}\right\}.$$
Given $u\in\mathcal{B}_{\delta}$, we denote 
$$f_u(x)=\tilde F_1(x,u(x)).$$
If $u\in\mathcal{B}_{d,e,\varepsilon}$ and $d,e,\varepsilon$ are small enough, then 
$$f_u\left(R_{d,e,\varepsilon}\right)\subset R_{d,e,\varepsilon}\quad \text{ and } \quad  f_u^n(x)\to0  \,\,\text{ for any } x\in R_{d,e,\varepsilon}.$$
Therefore, if we find $u\in\mathcal{B}_{d,e,\varepsilon}$ such that 
$$\tilde F_2(x,u(x))=u(\tilde F_1(x,u(x)))$$
then $\{(x,u(x)):x\in R_{d,e,\varepsilon}\}$ will be a parabolic curve of $\tilde F$. 
Consider $u\in\mathcal{H}_{d,e,\varepsilon}$ and denote $x_n=f_u(x_{n-1})$ for any $x_0\in R_{d,e,\varepsilon}$. 
If $d,e,\varepsilon>0$ are sufficiently small then the series
$$Tu(x_0)=E(x_0)^{-1}\sum_{n\ge0}\left[E(x_n)u(x_n)-E(x_{n+1})\tilde F_2(x_n,u(x_n))\right],$$ 
where $E(x)$ is the function defined by \eqref{eq:defE}, is normally convergent and defines an element $Tu\in\mathcal{H}_{d,e,\varepsilon}$; moreover, $T:\mathcal{H}_{d,e,\varepsilon}\to\mathcal{H}_{d,e,\varepsilon}$ is a contraction. Observe that $T$ coincides with Hakim's operator if $p=0$, since in this case $E(x)=x^{A_0}$. Since the solutions $u\in\mathcal{H}_{d,e,\varepsilon}$ of the invariance equation above are exactly the fixed points of $T$, applying Banach fixed point theorem we find, for $d,e,\varepsilon$ small enough,  an element $u\in \mathcal{H}_{d,e,\varepsilon}$ such that $\{(x,u(x)):x\in R_{d,e,\varepsilon}\}$ is a parabolic curve for $\tilde F$, so its blow-down is a parabolic curve for $F$ (an alternative argument for the existence of a fixed point, which simplifies the technical computations, is the use of Schauder fixed point theorem as it is done in \cite{Lop-R-S-V}). Varying $m$, we can show that the orbits in the parabolic curve are asymptotic to $\tilde\Gamma$.

Assume now that $\ell$ is a node direction and define, for $d,e,\varepsilon>0$, 
$$S_{d,e,\varepsilon}=\left\{(x,y)\in\C^2:x\in R_{d,e,\varepsilon}, |y|<|x|\right\}.$$
If $d,e,\varepsilon$ are small enough then $\tilde F(S_{d,e,\varepsilon})\subset S_{d,e,\varepsilon}$ and moreover every orbit in $S_{d,e,\varepsilon}$ converges to 0 asymptotically to $\tilde\Gamma$.
The blow-down of $S_{d,e,\varepsilon}$ is then a parabolic domain for $F$ whose orbits converge asymptotically to $\Gamma$. This finishes the proof, taking $r=k+p$.
\end{proof}

\begin{remark}
It can also be shown that, with the hypotheses of Theorem~\ref{th:LRRS}, every convergent orbit that is asymptotic to $\Gamma$ is eventually contained in $\Omega_1\cup...\cup\Omega_r$. Theorem~\ref{th:LRRS} also holds in higher dimension (see \cite{Lop-R-S-V}). The idea of the proof in the higher dimensional case is a generalization the one presented here: after a finite number of changes of coordinates, blow-ups with smooth centers and ramifications adapted to $\Gamma$, we can find coordinates $(x,\mathbf y)\in\C\times\C^{n-1}$ such that the strict transform $\tilde \Gamma$ of $\Gamma$ is non-singular and transverse to $(x=0)$ and the transform $\tilde F$ of $F$ has the form
\begin{align*}
\tilde F_1(x,\mathbf y)&=x-x^{k+p+1}+O(x^{2k+2p+1})\\
\tilde F_2(x,\mathbf y)&=\exp\left(x^kD(x)+x^{k+p}C\right)\mathbf y+O(x^{k+p+1})
\end{align*}
where $k\ge1$, $p\ge0$, $D(x)$ is a diagonal matrix of polynomials of degree at most $p-1$ and $C$ is a constant matrix commuting with $D(x)$. As in the two-dimensional case, this expression allows to identify, for each attracting direction $\ell=\xi\R^+$ with $\xi^{k+p}=1$, which variables have a node behavior in that direction and which have a saddle behavior. Using a fixed point theorem we find, associated to each direction $\ell$, a parabolic manifold whose dimension is equal to the number of node variables in that direction plus one. Although we are only dealing with the tangent to the identity case, let us mention that the results obtained in \cite{Lop-R-R-S} and \cite{Lop-R-S-V} are more general and can be applied to any biholomorphism $F$ in $(\C^n,0)$ with a formal invariant curve $\Gamma$ not contained in the set of fixed points of $F$ and such that $\lambda=(F|_\Gamma)'(0)$ is either a root of unity or satisfies $|\lambda|<1$ (see \cite{Lop-R-S-V} for details).
\end{remark}

Concerning the first question stated at the beginning of this section, the answer is also essentially positive:

\begin{theorem}[López-Hernanz, Rosas \cite{Lop-R}]\label{th:LR}
Let $F$ be a tangent to the identity biholomorphism in $(\C^2,0)$ and let $[v]$ be a characteristic direction of $F$. Then either there is a curve of fixed points tangent to $[v]$ or there exist parabolic manifolds for $F$ where the orbits converge tangentially to $[v]$.
\end{theorem}

\begin{proof}
Let $\pi$ be the blow-up at the origin and let $p$ be the point that corresponds to $[v]$. Denote $X=\log F$, $E=\pi^{-1}(0)$ and $\tilde X=(\pi^*X)_p$. If the transform $\tilde F$ of $F$ at $p$ has a curve of fixed points we are done, so we assume that this is not the case. If $\tilde X$ has a formal invariant curve $\Gamma$ different from $E$, then $\Gamma$ is invariant for $\tilde F$, so by Theorem~\ref{th:LRRS} we obtain parabolic manifolds for $\tilde F$ whose blow-down are parabolic manifolds for $F$ and we are done. Otherwise, $\tilde X$ only has $E$ as invariant curve, so by Camacho-Sad theorem $E$ is invariant for $\sat \tilde X$ and then $X$ is non-dicritical. Then, since $[v]$ is a zero of $P_X$ and $\tilde F$ has no curves of fixed points, $\tilde X$ is strictly singular. In this case, we can use a result of Camacho, Lins Neto and Sad \cite{Cam-L-S} which guarantees that if $Y$ is a saturated singular formal vector field with a unique invariant curve, which moreover is non-singular, then there is a point $q$ in the reduction of singularities of $Y$ such that the saturation of the transform of $Y$ at $q$ has a saddle-node singularity. We find therefore a saddle-node singularity $q$ in the reduction of $\sat\tilde X$, which belongs to two components of the exceptional divisor since $\sat\tilde X$ has only $E$ as invariant curve. Then we can find coordinates $(x,y)$ at $q$ such that the transform of $\sat\tilde X$ at $q$ can be written
$$x^ry^m\left[x(a+A(x,y))\parcial x+(bx+B(x,y))\parcial y\right]$$
with $r,m\ge1$, $a\neq0$, $b\in\C$, $\nu_0(A)\ge1$, $\nu_0(B)\ge2$ and $x\nmid B$, so the transform $\hat F$ of $\tilde F$ at $q$ has the form
$$\hat F(x,y)=\left(x+x^{r+1}y^m\left[a+O(x,y)\right],y+x^ry^m\left[cy^{p+1}+O(x,y^{p+2})\right]\right)$$
for some $p\ge1$ and some $c\neq0$. This kind of biholomorphisms, that appear after a blow-up centered at a so-called irregular characteristic direction, were studied by Vivas in \cite{Viv}, and she showed that there exist at least $rp$ parabolic domains. The image of each of these domains by the reduction is a parabolic domain for $\tilde F$, finishing the proof of the theorem. 
\end{proof}

\begin{remark}
We do not know if Theorems~\ref{th:Abate} and \ref{th:LR} hold in higher dimension. Consider a tangent to the identity biholomorphism $F$ in $(\C^n,0)$ with $n\ge3$; as we have mentioned, non-degenerate characteristic directions of $F$ support parabolic curves and formal invariant curves of $F$ support either curves of fixed points or parabolic manifolds. Therefore, two possible strategies to deal with this higher dimensional case are the following: 
\begin{enumerate}[$\circ$]
\item If $F$ has a formal invariant curve, then the existence of curves of fixed points or parabolic manifolds follows, generalizing Theorem~\ref{th:Abate}. However, this is not always the case. Gómez-Mont and Luengo provided in \cite{Gom-L}  examples of vector fields in $(\C^3,0)$ that do not have formal invariant curves, so their time-1 flow $F$ does not either. However, in these examples $F$ has parabolic curves, since after one blow-up there are non-degenerate characteristic directions (see \cite{Aba-T}).
\item If after some blow-ups we find a non-degenerate characteristic direction, the existence of parabolic curves is guaranteed. However, this is not possible in general. Mongodi and Ruggiero provide in \cite{Mon-R} an example of a tangent to the identity biholomorphism $F$ in $(\C^3,0)$ for which no non-degenerate characteristic directions appear after any number of (reasonable) blow-ups. However, their example has parabolic manifolds, since $F$ has formal invariant curves and no curves of fixed points.
\end{enumerate}
\end{remark}


\begin{thebibliography}{99}

\bibitem{Aba} \obra{Abate, M.}{The residual index and the dynamics of holomorphic maps tangent to the identity}{Duke Math. J. 107 (2001), no. 1, 173--207}

\bibitem{Aba-T} \obra{Abate, M.; Tovena, F.}{Parabolic curves in $\C^3$}
{Abstr. Appl. Anal. 2003, no. 5, 275--294}

\bibitem{Ari-R} \obra{Arizzi, M.; Raissy, J.}{On \'Ecalle-Hakim's theorems in holomorphic dynamics}{Frontiers in complex dynamics, 387--449, Princeton Math. Ser., 51, Princeton Univ. Press, Princeton, NJ, 2014}
              
\bibitem{Bal} \obra{Balser, W.}{From divergent power series to analytic functions}
{Lecture Notes in Mathematics, 1582. Springer-Verlag, Berlin, 1994}

\bibitem {Bri-B} \obra{Briot, C.; Bouquet, J.C.}{Recherches sur les propri\'{e}t\'{e}s des fonctions d\'{e}finies par des \'{e}quations diff\'{e}rentielles}{J. \'{E}cole Polytechnique 21, 1856, 133--198}

\bibitem{Bro-C-L} \obra{Brochero Martínez, F.E.; Cano, F.; López-Hernanz, L.} {Parabolic curves for diffeomorphisms in $\C^2$}{Publ. Mat. 52 (2008), no. 1, 189--194}

\bibitem{Cam-L-S} \obra{Camacho, C.; Lins Neto, A.; Sad, P.}{Topological
invariants and equidesingularization for holomorphic vector fields}{J. Differential Geometry 20 (1984) 143--174} 

\bibitem{Cam-S} \obra{Camacho, C.; Sad, P.}{Invariant varieties through singularities of holomorphic vector fields}
{Ann. of Math. (2)  115  (1982), no. 3, 579--595}
 
\bibitem{Cam-S2} \obra{Camacho, C.: Sad, P.}{Pontos singulares de equa\c c\~oes diferenciais anal\'{\i}ticas}{
16$\sp {\rm o}$ Col\'oquio Brasileiro de Matem\'atica. Instituto de Matem\'atica Pura e Aplicada (IMPA), Rio de Janeiro, 1987}

\bibitem{Can2} \obra{Cano, J.}{An extension of the Newton-Puiseux polygon construction to give solutions of Pfaffian forms}{Ann. Inst. Fourier (Grenoble)  43  (1993),  no. 1, 125--142}

\bibitem{Can} \obra{Cano, J.}{Construction of invariant curves for singular holomorphic vector fields}
{Proc. Amer. Math. Soc.  125  (1997),  no. 9, 2649--2650}

\bibitem{Can-C-D} \obra{Cano, F; Cerveau, D.; Deserti, J.}{Théorie élémentaire des feuilletages holomorphes singuliers}{Éditions Belin, 2013}

\bibitem{Dul} \obra{Dulac, H.}{Solutions d'un système d'équations différentielles dans le voisinage de valeurs singulières}
{Bull. Soc. Math. France 40, 324--383 (1912)}

\bibitem{Eca} \obra{\'Ecalle, J.}{Les fonctions r\'{e}surgentes. Tome III. L'\'{e}quation du pont et la classification analytique des objects locaux} {Publications Math\'{e}matiques d'Orsay, 85-5. Universit\'{e} de Paris-Sud, D\'{e}partement de Math\'{e}matiques, Orsay, 1985}

\bibitem{Fat} \obra{Fatou, P.}{Sur les \'{e}quations fonctionelles}{Bull. Soc. Math. France 47 (1919) 161--271}

\bibitem{Gom-L} \obra{Gómez-Mont, X.; Luengo, I.} {Germs of holomorphic vector fields in $\C^3$ without a separatrix}{Invent. Math. 109 (1992), no. 2, 211--219}

\bibitem{Hak} \obra{Hakim, M.}{Analytic transformations of $(\C^p,0)$ tangent to the identity}{Duke Math. J. 92 (1998), no. 2, 403--428}

\bibitem{Hak2} \obra{Hakim, M.}{Transformations tangent to the identity. Stable pieces of manifolds}{Pr\'epublication d'Orsay num\'ero 30, Orsay, 1997}

\bibitem{Huk-K-M} \obra{Hukuhara, M.; Kimura, T.; Matuda, T.}{Equations diff\'{e}rentielles ordinaires du premier ordre dans le champ complexe}{Publications of the Mathematical Society of Japan, 7. The Mathematical Society of Japan, Tokyo 1961}


\bibitem{Ily-Y} \obra{Ilyashenko, Y.; Yakovenko, S.}{Lectures on analytic differential equations}{Graduate Studies in Mathematics, 86. American Mathematical Society, Providence, RI, 2008}

\bibitem{Lea} \obra{Leau, L.}{\'{E}tude sur les \'{e}quations fonctionelles \`{a} une ou \`{a} plusieurs variables}{Ann. Fac. Sci. Toulouse Sci. Math. Sci. Phys. 11 (1897) 1--110}

\bibitem{Lop-R-R-S} \obra{López-Hernanz, L.; Ribón, J.; Sanz-Sánchez, F.; Vivas, L.}
{Stable manifolds of biholomorphisms in $\C^n$ asymptotic to formal curves}{Proc. Lond. Math. Soc. (3) 125 (2022), no. 2, 277--317}

\bibitem{Lop-R-S-V} \obra{López-Hernanz, L.; Raissy, J.; Ribón, J.; Sanz-Sánchez, F.}
{Stable manifolds of two-dimensional biholomorphisms asymptotic to formal curves}
{Int. Math. Res. Not. IMRN 2021, no. 17, 12847--12887}

\bibitem{Lop-R} \obra{L\'opez-Hernanz, L.; Rosas, R.}{Characteristic directions of two-dimensional biholomorphisms}{Compositio Mathematica, 156 (2020), no. 5, 869--880}

\bibitem{Lop-R2} \obra{López-Hernanz, L.; Rosas, R.}{A flower theorem in dimension two}{Ergodic Theory Dynam. Systems, 2025}

\bibitem{Lor} \obra{Loray, F.}{Pseudo-groupe d'une singularit\'e de feuilletage holomorphe en dimension deux}{Ensaios Matemáticos, 2021, Volume 36, 53--274}

\bibitem{Mat-M} \obra{Mattei, J.F.; Moussu, R.}{Holonomie et int\'{e}grales premi\`{e}res} {Ann. Sci. \'{E}cole Norm. Sup. (4)  13  (1980), no. 4, 469--523}

\bibitem{Mar-R} \obra{Martinet, J.; Ramis, J.P.}{Probl\`emes de modules pour des \'equations diff\'erentielles non
lin\'eaires du premier ordre}{Inst. Hautes \'Etudes Sci. Publ. Math., {\bf 55} (1982), 63--164}

\bibitem{Mil}\obra{Milnor, J.} {Dynamics in one complex variable}{Annals of Mathematics Studies, 160. Princeton University Press, Princeton, NJ, 2006}

\bibitem{Mon-R} \obra{Mongodi, S.; Ruggiero, M.}{Birational properties of tangent to the identity germs without non-degenerate singular directions}{J. Lond. Math. Soc. (2) 107 (2023), no. 4, 1242--1296}

\bibitem{Ram} \obra{Ramis, J.P.}{Les s\'{e}ries $k$-sommables et leurs applications} {Complex analysis, microlocal calculus and relativistic quantum theory (Proc. Internat. Colloq., Centre Phys., Les Houches, 1979),  pp. 178--199,
Lecture Notes in Phys., 126, Springer, Berlin-New York, 1980}

\bibitem{Rib} \obra{Rib\'{o}n, J.}{Families of diffeomorphisms without periodic curves}{Michigan Math. J.  53  (2005),  no. 2, 243--256}

\bibitem{Sei}\obra{Seidenberg, A.}{Reduction of the singularities of the differential equation $Ady=Bdx$}{Am. J. of Math., 90 (1968), 248--269}

\bibitem{Viv}\obra{Vivas, L.}{Degenerate characteristic directions for maps tangent to the identity}{Indiana Univ. Math. J. 61 (2012), no. 6, 2019--2040}

\end{thebibliography}
\end{document}